\documentclass[10pt,draft,reqno]{amsart}
     \makeatletter
     \def\section{\@startsection{section}{1}%
     \z@{.7\linespacing\@plus\linespacing}{.5\linespacing}%
     {\bfseries%\normalfont\scshape
     \centering
     }}
     \def\@secnumfont{\bfseries}
     \makeatother
\newtheorem{theorem}{Theorem}[section]
\newtheorem{lemma}[theorem]{Lemma}
\newtheorem{proposition}[theorem]{Proposition}

\theoremstyle{definition}
\newtheorem{definition}[theorem]{Definition}

\theoremstyle{remark}
\newtheorem{remark}[theorem]{Remark}
\numberwithin{equation}{section} 
\newcommand{\Real}{\mathbb{R}}

\newcommand{\R}{\mathbb{R}}

\begin{document}

\title[Entropy Chaos and BEC]{ Entropy Chaos and Bose-Einstein Condensation}

\author[S. Albeverio, F.C. De Vecchi and S. Ugolini]{\bf Sergio Albeverio$^1$, Francesco C. De Vecchi$^2$ and Stefania Ugolini$^2$}

\dedicatory{$^1$ Department of Applied Mathematics, University of
Bonn, HCM, BiBoS, IZKS; CERFIM, Locarno.
albeverio@iam.uni-bonn.de\\
$^2$ Dipartimento di Matematica, Universit\`a di Milano, via
Saldini 50, Milano. francesco.devecchi@unimi.it,
stefania.ugolini@unimi.it}

\begin{abstract}
We prove the entropy-chaos property for the system of N
undistinguishable interacting diffusions rigorously associated
with the ground state of N trapped Bose particles in the
Gross-Pitaevskii scaling limit of infinite particles. On the
path-space we show that the sequence of probability measures of
the one-particle interacting diffusion weakly converges to a limit
probability measure, uniquely associated with the minimizer of the
Gross-Pitaevskii functional.
\end{abstract}

\subjclass[2000] {Primary  60J60, 60K35, 81S20, 94A17; Secondary
26D15,81S20, 60G10,60G40}

\keywords{Bose-Einstein Condensation, Gross-Pitaevskii scaling
limit, Stochastic Mechanics, interacting Nelson diffusions,
entropy chaos, Kac's chaos, convergence of probability measures on
path space.}

\maketitle

%%% ----------------------------------------------------------------------
%\pagestyle{myheadings} \thispagestyle{myheadings}
%\markboth{S.Albeverio, F.C. De Vecchi and S.Ugolini}
%             {Entropy Chaos and Bose-Einstein Condensation }

%%% ------------------------------

\section {\bf Introduction}

Between 2000 to 2002 Lieb, Seiringer and Yngvason solved the
problem of giving a physical and mathematical justification of the
Gross-Pitaevskii (GP) model for a quantum mechanical particle in a
Bose-Einstein condensate (\cite{Lieb},\cite{Lieb2}, see also
\cite{LiebBook}). Starting from the N body Hamiltonian, describing
the system of N Bose particles in a suitable trapping potential
$V$ which interact through a pairwise potential $v$ (see Section
2), they proved that the GP mathematical model can be rigorously
obtained from the N body Hamiltonian by performing a suitable
limit of infinitely many particles together with a well-defined
re-scaling of the interaction potential $v$. In particular they
showed that in this GP scaling limit the one-particle quantum
mechanical energy converges to the minimum of the GP functional
and that the corresponding interaction energy asymptotically
localizes in the points where the other particles are. Moreover
they prove the complete or exact Bose-Einstein Condensation
(\cite{Lieb2}) in terms of the complete factorization of the
n-particles reduced density matrix.

More recently, there has been proposed a stochastic description of
a Bose-Einstein condensate (the first time in \cite{LM}) by using
the well-known Nelson map. In 2011 (\cite{MU}) it has been shown
that under the hypothesis of continuous differentiability of the
many body ground state wave function one can define a well-defined
one particle stochastic process which, in the GP scaling limit,
continuously remains outside a time dependent interaction-region
with probability one. When the one-particle process is suitably
stopped, its convergence in total variation sense can be proved
using the relative entropy approach. The convergence is towards a
limit diffusion process whose drift is uniquely determined by the
minimizer of the Gross-Pitaevskii functional, usually called wave
function of the condensate. Successively the phenomenon of the
asymptotic localization of relative entropy has been investigated
(\cite{MU1}) as a probabilistic counterpart of the asymptotic
localization of the interaction energy. In \cite{Ugolini}  Kac's
chaos was established for the symmetric probability law of the N
interacting diffusions system associated to the ground state of
the N body Hamiltonian under the GP scaling limit. Since the
Nelson map cannot be applied to a non linear Hamiltonian, the
problem of correctly individuating the process corresponding to
the minimizer of the (non linear) Gross-Pitaevskii functional had
to be faced in \cite{AU}. Performing a sort of Doob h-transform of
the GP Hamiltonian a non linear diffusion generator with a killing
rate governed by the wave function of the condensate was derived.
The density-dependent killing rate is the probabilistic way to
describe the self-interaction suffered from the usual diffusion
process with drift of gradient type by the generic other particle
which shares the same invariant probability density. With the
introduction of a proper one-particle relative entropy an
existence theorem for the probability measure
associated to the minimizer of the GP functional was proved in \cite{DeVU}.\\
In the present paper we focus on some other probability measures
convergence problems related to the GP scaling limit described
above. We prove that the entropy chaos property holds for the N
interacting diffusions system. This is a stronger chaotic property
than Kac's chaos property. The result is obtained, following
\cite{HaMi}, by using Kac's chaos result plus a regularity
condition on the trapping potential $V$ which guarantees a finite
moment condition on the corresponding density measures. While
Kac's chaos is expressed in terms of weak convergence of all the
n-marginal laws toward the n-fold tensor product of the asymptotic
probability density, the entropy chaos, which is given by the weak
convergence of the one-particle marginal measure plus the
convergence of the entropy functionals, allows to prove a total
variation convergence result for the same sequence of n-marginal
laws. This result concerns the fixed time marginal laws of our N
interacting diffusions on the product space $\R^{3N}$. Since  the
interaction between the particles asymptotically concentrates on a
random region having Lebesgue measure zero (see \cite{MU}) but it
does not disappear, the convergence problem of the one-particle
probability measure on the path space is not trivial and one
cannot hope to find a convergence result as strong as the total
variation one. In \cite{MU}, in fact, the latter has been proved
only for the stopped version of the one-particle non-markovian
diffusion. In this paper we show that on the path space a weak
convergence result can be obtained for the one-particle diffusion
process by taking advantage of some relevant properties of the
corresponding sequences of stopping times. In spite of the strong
coupling between the probability measure and the associated
stopping time the crucial property is that, for all $t>0$, the
probability that the actual stopping time is larger than t
converges
to one in the infinite particles limit (\cite{MU}).\\
In Section 2  Carlen's class of N interacting diffusions
rigorously associated with the ground state of the N body
Hamiltonian for N Bose particles is
presented.\\
In Section 3 the main analytical results obtained by Lieb,
Seiringer
and Yngvason for the specific quantum problem are briefly recalled.\\
Kac's chaos result for the sequence of N probability laws is
described in Section 4.\\
In Section 5 we discuss the concept of entropy chaos and its
relation with Kac's chaos. We prove that under an assumption on
the form of the trapping potential $V$  the entropy chaos holds
for our sequence of N probability laws. Moreover under the
hypothesis of convexity of the trapping potential $V$ we establish
an inequality of HWI type, i.e. between the relative entropy (H)
and the relative Fisher information (I) through the Wasserstein
distance (W), which also allows to prove the entropy chaos
property. Finally we show that the entropy chaos implies the total
variation convergence for the same sequence of probability
measures. In Section 6 we establish the weak convergence of the
one-particle interacting diffusion on the path space under the GP
scaling limit. With this result a probabilistic justification of
the GP mathematical model for the condensate is also provided.

   %%%%%%%%%%%%%%%%%%%%%%%%%%%%%%%%%%%%%%%%%%%%%%%%%%%%%%%%%%%%%%%%%%%%%%%%%%%%%%%%%%%%%%%%%I

\section{ Nelson-Carlen Diffusions and Bose-Einstein Condensation}

\vskip5pt

Nelson's Stochastic Mechanics is an alternative formulation of
Quantum Mechanics which allows to study quantum phenomena using a
well determined class of diffusion processes
(\cite{Nelson1},\cite{Nelson2},\cite{Carlen},\cite{Carlen2}). See
\cite{CarlenN} for a more recent review on Stochastic Mechanics.

We will briefly introduce the class of \textit{Nelson} diffusions
which are associated to a solution of a Schr\"{o}dinger equation.

Let the complex-valued function (\textit{wave function})
$\psi(x,t)$ be a solution of the equation:
\begin{equation}
i \partial_t \psi(x,t)=H\psi(x,t),\quad t\in {\R}, \quad x\in
{\R}^d,
\end{equation}
\noindent  with $\psi(x,0)=\psi_0(x)$, corresponding to the
Hamiltonian operator:
$$ H=-\frac{\hbar^2}{2m}\triangle +V(x),$$
\noindent where $\hbar$ denotes the reduced Planck constant, $m$
denotes the mass of a particle, and $V$ is some scalar potential.

Let us set:
\begin{equation}\label{osmvelocity}
u(x,t):=Re[\frac{\nabla\psi(x,t)}{\psi(x,t)}]
\end{equation}
\begin{equation}\label{curvelocity}
v(x,t):=Im[ \frac{\nabla\psi(x,t)}{\psi(x,t)}]
\end{equation}
\noindent when $\psi(x,t)\neq 0$ and, otherwise, set both $u(x,t)
$ and $v(x,t) $ to be  equal to zero.  Let us put
\begin{equation}
b(x,t):=u(x,t)+v(x,t)
\end{equation}

In a more general approach Carlen (\cite{Carlen3}) introduced the
 following diffusions class,  mainly characterized in terms of \textit{proper infinitesimal
 characteristics} $(\rho_t(x),v_t(x))$
consisting of a time-dependent probability density $\rho_t$ and a
time-dependent vector field $v_t(x)$ defined $\rho_t(x)dxdt-a.e$,
so constructed as to  have the time-reversal symmetry.

The  pairs $(\rho_t,v_t)$ are such that:
$$
\int_{\Real^d}f(x,T)\rho(x,T)dx-\int_{\Real^d}f(x,0)\rho(x,0)dx=\int_0^T\int_{\Real^d}(v_t\cdot\nabla
f)(x,t)dx$$ \noindent for all $T\geq 0$ and all $f\in
C_0^{\infty}(\Real^{d+1})$

Let $(\Omega, \mathcal F, \mathcal F_t,X_t) $, with
$\Omega=C(\Real_+,\Real^d)$, be the evaluation stochastic process
$X_t(\omega)=\omega(t)$, with $\mathcal F_t=\sigma(X_s, s\leq t)$
the natural filtration.

Carlen (\cite{Carlen},\cite{Carlen2},\cite{Carlen3}) proved that
if $(\rho_t,v_t)$ is a proper infinitesimal characteristic and if
the following \textit{finite energy condition} holds:
$$\int_0^T(||\nabla\sqrt{\rho_t}||^2_{L^2}+||v_t\sqrt{\rho_t}||^2_{L^2})dt
<+\infty,$$ \noindent with $||\cdot||_{L^2}$ the $L^2(\R^d\times
\R, dxdt)$-norm, for all $T\geq 0$ and all $f\in
C_0^{\infty}(\Real^{d+1})$, then there exists a unique Borel
probability measure $\mathbb P$ on $\Omega$ such that

i)$(\Omega, \mathcal F, \mathcal F_t,X_t,\mathbb P) $ is a Markov
process;

ii) the image of $\mathbb P$ under $X_t$ has density
$\rho(t,x):=\rho_t(x)$;

iii) $W_t:=X_t-X_0-\int_0^tb(X_s,s)ds$

\noindent is a $(\mathbb P,\mathcal F_t)$-Brownian Motion.

\vskip10pt In this Carlen diffusions class the Nelson diffusions
are properly those having the  pairs $(\rho_t,v_t)$ of the
following form:
$$\rho_t=\psi_t\bar{\psi_t}\quad \quad v_t=Im[ \frac{\nabla\psi_t}{\psi_t}]$$
\noindent ($\bar{\psi_t}$ being the conjugate complex function to
$\psi_t$). \vskip 10pt The continuity problem for the above
Nelson-Carlen map (from solutions of Schr\"{o}dinger equations to
probability measures on the path space given by the corresponding
Nelson-Carlen diffusions) is investigated in \cite{DPos}. For a
generalization to the case of Hamiltonian operators with magnetic
potential see \cite{PosU}.

From now on we will mainly consider the case where $d=3$.

We adopt the following notations: capital letters for stochastic
processes or, otherwise, we will explicitly specify them, $\hat
Y=(Y_1,...,Y_N)$ to denote arrays in $\R^{3N}$, $N\in {\mathbb
N}$, and bold letters for vectors in $\R^3$.

\vskip4pt

The Hamiltonian introduced to describe the recent experiments
(\cite{KeDr}, \cite{Cornell}) on BEC is the following N-body
Hamiltonian

\begin{equation} \label{HN}
H_N=\sum _{i=1}^{N}(-\frac {\hbar^2}{2m}{\triangle}_{i}+V({\bf
r_i}))+ \sum _{1\leq i< j \leq N} v(\bf r_i- \bf r_j)
\end{equation}
\noindent where $V$ is a confining potential,  $v$ a pair-wise
repulsive interaction potential and ${\bf r}_i \in \R^3,
i=1,...,N$. It operates on symmetric wave functions $\Psi$ in the
complex $L^2(\Real^{3N})$-space in order to satisfy the symmetry
permutation prescription for Bose particles.

We consider the mean quantum mechanical energy
\begin{equation}\label{quantenergy}
E[\Psi]=T_{\Psi}+ \Phi_{\Psi}
\end{equation}
where
$$T_{\Psi}=\sum _{i=1}^{N}\int_{\R^{3N}}|\nabla_i\Psi|^2d{\bf r}_1\cdot\cdot\cdot d{\bf r}_N$$
is physically called the \textit{kinetic energy} and
$$\Phi_{\Psi}= \sum _{i=1}^{N}\int_{\R^{3N}}V({\bf r}_i)|\Psi|^2d{\bf r}_1\cdot\cdot\cdot d{\bf r}_N
+\frac{1}{2}\sum _{i=2}^{N}\int v({\bf r}_1-{\bf
r}_i)|\Psi|^2d{\bf r}_1\cdot\cdot\cdot d{\bf r}_N$$ the
\textit{potential energy} associated with $\Psi$. The variational
problem associated to $H_N$ consists in minimizing $E [\Psi]$ with
respect to the complex-valued function $\Psi$ in $L^2(\R^{3N})$
subject to the constrain $\|\Psi\|_2=1$. If such a minimizing
function $\Psi^0_N$ exists it is called a \textit{ground state}.
The corresponding energy $E_0[\Psi^0_N]$ given by
$$E_0[\Psi^0_N]:=\inf\{E(\Psi): \int |\Psi|^2=1\}$$
is known as \textit{ground state energy}.

\noindent Under suitable assumptions on the potentials $V$ and $v$
one can prove the existence of the ground state $\Psi^0_N$ for
\eqref{HN}. Uniqueness of the ground state is to be understood as
uniqueness  apart from an \textit{overall phase}. For our purposes
we need a strictly positive and continuous differentiable ground
state. See \cite{ReedSimon} (Thm.XIII.46 and XIII.47) for the
regularities conditions on the potentials $V$ and $v$ implying the
strictly positivity and \cite{ReedSimon} (Thm.XIII.11) for those
implying the differentiability of the ground state wave function.

Here we precisely identify the interacting diffusions system
rigorously associated to the ground state solution $\Psi^0_N$ of
the Hamiltonian \eqref{HN}.

In this case, in fact, the pairs of proper infinitesimal
characteristics are of the form

$$(\rho_N,0) \quad \quad \rho_N:=|\Psi^0_N|^2\quad \quad \sqrt{\rho_N}\in
H_1(\Real^{3N})$$ with $H_1$ the Sobolev space of functions with
square integrable generalized derivatives.\\
Introducing the probability space $(\Omega^N, \mathcal F^N,
\mathcal F^N_t,\hat Y_t) $, with $\hat Y_t(\omega)=\omega(t)$ the
evaluation stochastic process, with $\mathcal F^N_t=\sigma(Y_s,
s\leq t)$ the natural filtration, then by Carlen's Theorem there
exists a unique Borel probability measure ${\mathbb P}_N$ such
that

i)$(\Omega^N, \mathcal F^N, \mathcal F^N_t,\hat Y_t,\mathbb P_N) $
is a Markov process;

ii) the image of $\mathbb P_N$ under $\hat Y_t$ has density
$\rho_N({\bf r})$;

iii) $\hat W_t:=\hat Y_t-\hat Y_0-\int_0^tb_N(\hat Y_s)ds$

\noindent where
$$b_N(\hat Y_t)=\frac
{\nabla^{(N)}\Psi^0_N}{\Psi^0_N}=\frac{1}{2}\frac
{\nabla^{(N)}\rho_N}{\rho_N}.$$

The stationary probability measure ${\mathbb P}_N$ with density
$\rho_N$ can be alternatively defined as the one associated to the
Dirichlet form (\cite{Fukushima},\cite{Fukushima2}, \cite{MR}):

\begin{equation}
\epsilon_{\rho_N}(f,g):=\frac{1}{2}\int_{\Real^{3N}}\nabla f({
r})\cdot\nabla g({ r})\rho_Nd{ r}^{3N}\quad \quad f,g \in
C_c^{\infty}(\Real^{3N})
\end{equation}

\vskip10pt

When Bose-Einstein condensation occurs, the condensate is usually
described by the order parameter $\phi_{GP}\in L^2(\R^3)$, also
called wave function of the condensate, which is the minimizer of
the Gross-Pitaevskii functional

\begin{equation} \label{GPfunc}
 E^{GP}[\phi] = \int  (\frac {\hbar^2}{2m}|\nabla \phi({\bf r})|^2 + V(r)|\phi({\bf r})|^2 +  g  |\phi({\bf r})|^4)d{\bf r}
 \end{equation}
 \noindent under the $L^2$-normalization condition
  $$ \int_{\R^3} |\phi^{GP}({\bf r})|^2d{\bf r}=1$$
  and where $g>0$ is a parameter depending on the interaction
  potential $v$ (see also assumption h3) in Section 3 below).
  Therefore $\phi_{GP}$ solves the stationary cubic non-linear equation (called Gross-Pitaevskii
equation)(\cite{Gross},\cite{Pitae})

 \begin{equation}\label{GP}
  -\frac {\hbar^2}{2m}\triangle \phi +V\phi + 2g |\phi|^2 \phi =\lambda
\phi
 \end{equation}
  \noindent $\lambda$, the real-valued Lagrange multiplier of the normalization constraint, is usually called  chemical potential. One can
prove that $\phi_{GP}$ is continuously differentiable and strictly
positive (\cite{Lieb}).

In \cite{LM} the stochastic quantization approach for the system
of $N$ interacting Bose particles has been exploited for the first
time.

It is proved in \cite{MU} that the Stochastic Mechanics of the
N-body problem associated to $H_N$ uniquely determines a well
defined stochastic process which describes the motion of the
single particle in the condensate, in the case of the
Gross-Pitaevskii scaling limit as introduced in \cite{Lieb}, which
allows to prove the existence of an exact Bose-Einstein
condensation for the ground state of $H^N$
(\cite{Lieb}\cite{Lieb2}). For the time-dependent derivation of
the Gross-Pitaevskii equation see \cite{Adami} and \cite{Erdos}.

   %%%%%%%%%%%%%%%%%%%%%%%%%%%%%%%%%%%%%%%%%%%%%%%%%%%%%%%%%%%%%%%%%%%%%%%%%%%%%%%%%%%%%%%%%I
\section{ Mean energy rescaling according to the GP limit }

\noindent For simplicity of notations, let us put $\hbar=2m=1$.

We consider the mean energy \eqref{quantenergy} expressed in terms
of the joint probability density of our $3N-$dimensional process
${\hat Y}$ as:
$$E_0[\rho_N]=E\{\sum_{i=1}^N[b_i^2(\hat Y)+V(Y_i(t))]+
\sum_{1\leq i< j\leq N} v(Y_i(t)-Y_j(t))\}$$ \noindent $b_i$ being
the drift of the interacting i-th particle, whose position is
given by the process $Y_i$. \vskip5pt

\noindent Following \cite{Lieb}, we assume

h1) $V(|\bf{r}_i|)$ is locally bounded, positive and going to
infinity when $|\bf{r}_i|$ goes to infinity.

h2) $v$ is smooth, compactly supported, non negative, spherically
symmetric, with finite positive \textit{scattering length} $a$
(\cite{LiebBook} Appendix C).

\noindent We perform the following scaling, known as
Gross-Pitaevskii (GP) scaling (\cite{Lieb}), writing
   \vskip 5mm

h3) $$ v(r)= v_1(\frac r a)/a^2 $$

$$
a=\frac{g}{4\pi N}
$$

\noindent where $v_{1}$ has scattering length equal to $1$ and
remains fixed while $N\uparrow +\infty$. Moreover $g
> 0$ as a consequence of our assumptions on $v$. \vskip 5pt

In \cite{Lieb} the following important theorem is proven.

\vskip 5mm

\begin{theorem}\label{theorem1} Under the
previous hypothesis h1),h2) h3) one has
\begin{equation} \label{LimE}
\lim_{N\rightarrow \infty}\frac{E_0[\rho_N]}{N}=E^{GP}[\rho_{GP}]
\end{equation}
and
\begin{equation} \label{Limrho}
\lim_{N\rightarrow \infty}\int \rho_Nd{\bf r}_2\cdot\cdot\cdot{\bf
r}_N =\rho_{GP}
\end{equation}
where $\rho_{GP}:=|\phi_{GP}|^2$, with $ \phi_{GP}$ the minimizer
of the Gross-Pitaevskii functional \eqref{GPfunc} and the
convergence is in the weak $L^1(\Real^3)$ sense.
\end{theorem}

\begin{remark} The one-particle marginal density $\rho^{(1)}_N$
converges weakly to $\rho_{GP}$ in the sense that the probability
measures $\rho^{(1)}_Nd{\bf r}$ weakly converge as $N\rightarrow
\infty$ towards the probability measure $\rho_{GP}d{\bf r}$ on
$\R^3$.
\end{remark}

Using some variational theorems (see \cite{CS}) the authors of
\cite{Lieb} are also able to uniquely characterize the limit of
the single components of the ground state energy $E_0[\rho_N]$.

\begin{theorem}\label{theorem2}(\textit{Energy Components}) Under the same hypothesis
h1),h2),h3) as in Theorem 1, let $\phi_0$ denote the solution of
the zero-energy \textit{scattering equation} for $v$:
\begin{equation}\label{scattering}
-\triangle \phi_0({\bf r})+{1\over 2} v({\bf r})\phi_0({\bf r})=0
\end{equation}
under the boundary condition $lim_{|{\bf r}|\rightarrow
+\infty}\phi_0({\bf r})=1$ \footnote{Setting $\phi_0({\bf
r})=\frac{u(r)}{r}$ with $r=|{\bf r}|$ the equation
\eqref{scattering} is equivalent to $-\triangle u( r)+{1\over 2}
v(r)u(r)=0$. The solution of this last equation with $u(0)=0$, for
r larger than the range of $v$, has the form: $u(r)=const(r-a)$
with $a> 0$  the scattering length of $v$. As a consequence
$\phi_0({\bf r})=const(1-\frac{a}{r})$. Thus for $r\uparrow
+\infty$  we have $lim_{|{\bf r}|\rightarrow +\infty}\phi_0({\bf
r})=1$.}. Putting
$$\hat{s}=\int |\nabla\phi_0|^2/(4\pi
a)$$ with $a$ as in h3), then $\hat{s}\in (0,1]$ and, recalling
that under the assumptions in \cite{Lieb} $|\nabla_1 \Psi^0_N({\bf
r}_1,...,{\bf r}_N)|^2$ exists and it is in $L^1(\R^{3N})$, we
have:
\begin{multline}\label{E1}
\lim _{N\uparrow \infty} \int_{\R^{3}} \int_{\R^{3N-3}} |\nabla_1
\sqrt{\rho_N}({\bf r}_1,...,{\bf r}_N)|^2d{\bf r}_1
\cdot\cdot\cdot d{\bf
r}_N=\int_{\R^3} |\nabla \sqrt{\rho_{GP}}({\bf r})|^2 d{\bf r}+ \\
+g\hat{s}\int_{\R^{3}}|\rho_{GP}({\bf r})|^2d{\bf r}
\end{multline}
and, moreover,
\begin{equation}\label{E2}
\lim _{N\uparrow \infty} \int_{\R^{3}} \int_{\R^{3N-3}}  V({\bf
r}_1)\rho_N({\bf r}_1,...,{\bf r}_N) d{\bf r}_1 \cdot\cdot\cdot
d{\bf r}_N= \int_{\R^{3}} V({\bf r})\rho_{GP}({\bf r}) d{\bf r}
\end{equation}

\begin{equation}\label{E3}
\lim _{N\uparrow \infty}{1\over2}\sum_{j=2}^{N} \int_{\R^{3}}
\int_{\R^{3N-3}}  v(|{\bf r}_1-{\bf r}_j|)\rho_N({\bf
r}_1,...,{\bf r}_N)d{\bf r}_1\cdot\cdot\cdot d{\bf r}_N=\\
(1-\hat{s})g\int_{\R^{3}} |\rho_{GP}({\bf r})|^2 d{\bf r}
\end{equation}
(with $g>0$ as in \eqref{GPfunc}).
\end{theorem} \vskip 10pt

\begin{remark} As remarked in \cite{MU}, the above assumptions on
$\Psi^0_N$ are satisfied under regularity assumptions on $V$ and
$v$ (\cite{ReedSimon} Sect. XIII 11).
\end{remark}
The next theorem states that the $L^2$-distance of the gradient
type drifts asymptotically goes to zero if we leave out the
neighborhoods of the points where the interaction is localized.
\vskip 10pt

\begin{theorem}\label{theorem3} (\textit{Energy Localization})
(\cite{Lieb2}). Defining

\begin{equation}
F^N({\bf r}_2 ,\dots ,{\bf r}_N):= (\bigcup _{i=2}^N  B ^N({\bf
r}_i))^c
\end{equation}

\noindent where $B^N({\bf r})$ denotes the open ball centered in
${\bf r}$ with radius $N^{-{1\over3}-\delta}$ where $0< \delta
\leq {4\over{51}}$,

\begin{equation} \lim _{N\uparrow \infty} \int_{\R^{3(N-1)}}d{\bf r}_2\cdot\cdot \cdot d{\bf r}_N
\int_{F^N({\bf r}_2 ,\dots ,{\bf r}_N)}(\frac{1}{2}\frac{\nabla_1
\rho_N}{\rho_N}-\frac{1}{2}\frac{\nabla_1 \rho_{GP}}{\phi_{GP}}
)^2\rho_{N}d{\bf r}_1=0
\end{equation}
\end{theorem}
\vskip5pt

\vskip10pt
\section{Kac's chaos in Bose-Einstein Condensation}

In this section we put $E=\R^3$ and we consider the symmetric
probability law $G^N$ of our $N$ interacting diffusions on the
product space $E^N$.

We describe some results obtained in \cite{Ugolini} concerning the
asymptotic behavior of our $N$ interacting diffusions
$(Y_1,Y_2,...,Y_N)$. The fixed time joint probability density of
$(Y_1,\dots,Y_N)$ is given by $\rho_N:=|\Psi^0_N|^2$, which is
invariant under spatial permutations. In \cite{MU} it has been
proved that if $\Psi^0_N$ is the ground state of $H_N$ (as
described in section 2) and it is strictly positive and of class
$C^1$, then the three-dimensional processes $\{
Y_i\}_{i=1,\dots,N}$ are equal in law.

We recall the non trivial \textit{chaotic property} introduced by
\cite{Kac}. It properly formalizes the fact that the random
variables $(Y_1,Y_2,...,Y_N)$ are becoming asymptotically an
independent random vector. This is described in terms of the
asymptotic factorization of the corresponding symmetric
probability laws ${G}_N$. One can see this property as the
probabilistic counterpart of the {\it complete Bose-Einstein
Condensation} (for the latter result see \cite{Lieb2}). Here we
consider only the probabilistic setting.

\begin{definition}[\cite{Sznitman}]\label{definition1}({\bf Kac's
chaos})

We say that ${G}_{N}$ is $G$-\textit{Kac's chaotic} (or,
equivalently, is chaotic in the Boltzmann's sense) if
\begin{equation}
\forall n\geq 1,\quad {G}^{N}_{n} \rightharpoonup { G}^{\otimes
n}, \quad N\uparrow \infty
\end{equation}
where ${G}^{N}_{n}$ stands for the n-th marginal of ${G}_{N}$ and
the convergence is the weak convergence of probability measures in
the marginal space $E^n=\R^{3n}$.
\end{definition}

Following \cite{HaMi}, we can reformulate Kac's chaos using the
Monge-Kantorovich-Wasserstein (MKW) transportation distance or
Wasserstein distance of order $1$ between ${G}^{N}_{n}$ and the
n-fold tensor product ${G}^{\otimes n}$.\\

\begin{definition}\label{definition2}({\bf MKW distance})

Given a bounded distance $d_E$ on $E=\R^3$ we introduce the
normalized distance $d_{E^n}$ given by
\begin{equation}
\forall X=(x_1,\dots,x_n), Y=(y_1,\dots,y_n) \quad
d_{E^n}:=\frac{1}{n} \sum_{i=1}^{n}d_E(x_i,y_i),
\end{equation}
and we define the MKW distance $W_1$ by
\begin{equation}
W_1(\mu_1,\mu_2):=\inf_{\pi \in \Pi(\mu_1,\mu_2)} \int_{E\times E}
d_E(x,y)\pi (dx,dy)
\end{equation}
where $\Pi(\mu_1,\mu_2)$ is the set of probability measures with
first marginal $\mu_1$ and second marginal $\mu_2$.
\end{definition}

\begin{definition}\label{definition3}(\textbf{ Kac's chaos in terms of
$W_1$})

With the same notations as in Definition \ref{definition1},
${G}_{N}$ is $G$-Kac's \textit{chaotic} if, and only if,
\begin{equation}
\forall n\geq 1,\quad W_1({G}^{N}_{n}, {G}^{\otimes
n})\longrightarrow 0 \quad as \quad N\uparrow \infty
\end{equation}
\end{definition}

\noindent In \cite{Ugolini} the following has been established

\vskip5pt

\begin{theorem}\label{theorem4}({\bf Kac's chaos for the measure  ${G}_{N}$})
%proposition 1
Under the hypothesis h1), h2), h3) (Section 3), the symmetric law
${G}_N$ is $G$-Kac's \textit{chaotic} (in the sense of Definition
\ref{definition1}).
\end{theorem}

With the usual notation
$$< \mu,\phi>=\int \phi(x)\mu(dx)$$
for a probability measure $\mu$ on $E$ and $\phi \in C_b(E)$, we
recall that there is another possible formulation of Kac's chaos
in terms of the empirical measures of our interacting diffusions
system.

\vskip5pt

Let us introduced the \textit{empirical measure} on $E$ associated
to an $E^N$-valued random vector ${\hat Y}=(Y_1,Y_2,...,Y_N)$:
$$\mu^N_{\hat Y}(d{\bf r}):=\frac{\sum_{i=1}^{N}\delta_{Y_i}(d{\bf r})}{
N}$$

First we recall the definition of chaos in terms of the empirical
measure.

\begin{definition}\label{definition4}({\bf Chaos by the empirical
measure})

We say that the exchangeable random vector ${\hat
Y}=(Y_1,\dots,Y_N)$ is $G$-chaotic if $\mu^N_{\hat Y}$ converges
in law to the constant (i.e. deterministic) random variable $G$.
\end{definition}

\noindent Following (\cite{Sznitman}) one can shown that the
\textit{chaotic} property in Theorem \ref{theorem2} implies a non
trivial convergence result for the \textit{empirical measure} (see
also \cite{Ugolini}).

\begin{proposition}\label{proposition1}
%proposition 1
If Kac's chaos holds for the probability law $G_N$, then the
\textit{empirical measure} $\mu^N_{\hat Y}(d{\bf
r}):=\frac{\sum_{i=1}^{N}\delta_{Y_i}(d{\bf r})}{ N}$ converges in
law as $N\uparrow \infty$ to the constant random variable $G$. In
particular one has for $N\uparrow +\infty$ and $\forall \phi\in
C_b(E)$:
$$E_{\rho_N}[(<\mu^N_{\hat Y}-G,\phi>)^2]\rightarrow 0$$
\end{proposition}

It is well-known that the chaotic property according to Definition
\ref{definition4} is really an equivalent formulation of Kac's
chaos as given in Definition \ref{definition1} (see
\cite{Sznitman}). For a detailed study of the quantitative
dependence between the cited different formulations of Kac's chaos
see \cite{HaMi}.

We conclude by stressing that the complete Bose-Einstein
Condensation proved in \cite{Lieb2} implies that  Kac's chaos
holds for the associated $N$ interacting diffusions system in the
GP scaling limit, under the same regularities conditions h1) h2)
on the potentials (\cite{Ugolini}).

The weak convergence of $\rho^{(1)}_Nd{\bf r}_1$ to
$\rho_{GP}d{\bf r}_1$ comes from Theorem \ref{theorem1}. It is
well-known that Kac's chaos is essentially (at least) given by
\begin{equation}\label{twomarginal}
\rho^{(2)}_N({\bf r}_1,{\bf r}_2)d{\bf r}_1d{\bf r}_2
\rightharpoonup \rho_{GP}({\bf r}_1)\rho_{GP}({\bf r}_2)d{\bf
r}_1d{\bf r}_2
\end{equation}
This result comes from the proof of the \textit{complete} BEC
given in \cite{Lieb2}. In particular it may be derived by reducing
to the diagonal subspace in the convergence of the 4-particle
reduced density matrices (see \cite{Ugolini}). The fact that
\eqref{twomarginal} is sufficient for having Kac's chaos is a
quite classical result (see \cite{Sznitman}, Proposition 2.2 and
\cite{HaMi}, Theorem 2.4). The counterpart of this property in the
standard analytical framework of BEC is proved in (\cite{Lieb2}
and \cite{LiebBook}, Remark after Theorem 7.1) and in
(\cite{Michelangeli}, Theorem 7.1.1). \vskip5pt

\section{Entropy chaos in Bose-Einstein Condensation}

In this section we describe the recent concept of entropy chaos,
first introduced in \cite{Carlen4} and then recently investigated
in \cite{HaMi}.

First we give the definition of entropy, in terms of a finite
moment condition, for probability  measures on the product space
$E^n$ (see \cite{HaMi}). The entropy, in fact, may not be
well-defined for probability measures decreasing two slowly at
infinity.

\begin{definition}\label{definition5}({\bf Entropy})
The entropy associated with a given measure ${G}_{N}$, admitting a
probability density $\rho_N\in L^1(E^N)$ such that $M_k(\rho_N) <
\infty$ for some $k>0$ ($M_k$ denoting the k-th moment), is given
by
\begin{equation}\label{definitionentropy}
\hat{H}(\rho_{N}):=\int_{E^N} \rho_{N} log({\rho_{N}})\\
=\int_{E^N} (\frac{\rho_{N}}{H_k}
log(\frac{\rho_{N}}{H_k})-\frac{\rho_N}{H_k}+1)H_k + \int_{E^N}
{\rho_{N}}log{H_k}
\end{equation}
with $H_k:=C_k\exp{(-|{\bf r}_1|^k-...-|{\bf r}_N|^k)}$, where
$C_k$ are normalization constants such that $H_k$ are probability
measures on $E^N$.
\end{definition}

Since $\rho_N$ and $H_k$ are probability measures on $E^N$, the
term on the right hand side is well-defined by the fact that the
first integral has a non negative integrand and the second one is
finite by the assumption of a finite k-th moment. Following
\cite{HaMi} we define the normalized entropy functional
\begin{equation}
H(\rho_N):=\frac{1}{N}\hat{H}(\rho_N)
\end{equation}
and we introduce the notion of entropy chaos.

\begin{definition}\label{definition6}({\bf Entropy chaos})
A sequence ${G}_{N}$ is $G$-entropy chaotic if
\begin{equation}\label{entropychaos}
{G}^{N}_1 \rightharpoonup G, \quad H({G}_{N}) \rightarrow H(G) <
+\infty,
\end{equation}
where $\rightharpoonup$ stands for weak convergence of measures,
as $N\rightarrow +\infty$.
\end{definition}

The next theorem (see \cite{HaMi}, Theorem 1.4) states that the
entropy chaos is a stronger property than Kac's chaos.

\begin{theorem}[\cite{HaMi}]\label{theorem5}
If a sequence ${G}_{N}$ is $G$-entropy chaotic, then ${G}_{N}$ is
$G$ Kac's chaotic.
\end{theorem}

\begin{proof}
We report a sketch of the proof for completeness ( for details see
\cite{HaMi}). By definition, the entropy-chaos means that
\eqref{entropychaos} holds. We  want to prove that for all $n\geq
1$ one has that ${G}^{N}_n \rightharpoonup G^{\otimes n}$. Fixing
$n\geq 1$, since $G^N_n$ is bounded in $E^n$, there exists a
subsequence $G^{N'}_n$ such that ${G}^{N'}_n \rightharpoonup F_n$
where $F_n$ is a probability measure on $E^n$. By the
super-additive property of the non normalized entropy $\hat{H}$
(defined according to Definition 5) and by taking the limit using
that $\hat{H}$ is lower semi-continuous and bounded by below one
has
\begin{equation}
\hat{H}(F_n)\leq liminf \hat{H}(G^{N'}_n)\leq liminf
\hat{H}(G^{N}_n)
\end{equation}
Dividing by $N$ and using the convergence of the entropies one
obtains
\begin{equation}
{H}(F_n)\leq liminf {H}(G^{N}_n)=H(G)
\end{equation}
Since the first marginal of $F_n$ is $F^1_n=G$, $G$ being the
limit of $G^N_1$ as $N\uparrow \infty$, and in general, by the
properties of the entropy, one has $H(F_n)\geq H(F^1_n)$ , then
one necessarily obtains $H(F_n)= H(G)$, which implies
$F_n=G^{\otimes n}$ a.e.. Since for our subsequence $G^{N'}_n$ we
have identified the limit $G^{\otimes n}$, we have proved that the
whole  sequence $G^{N}_n$ weakly converges to $G^{\otimes n}$ as
$N\uparrow \infty$.
\end{proof}

Taking advantage of some recent results in \cite{HaMi} we are able
to prove that the entropy-chaos holds for our sequence of
probability laws ${G}_{N}$ under some regularity assumptions on
the trapping potential $V$.

In order to achieve this, it is useful to introduce also the MKW
distance of order 2.

\begin{definition}\label{definition7}({\bf MKW distance of order
2})

The MKW distance of order 2, $W_2$, is defined as the MKW distance
of order 1 with the following choice for the normalized distance
$d_{E^n}$
\begin{equation}
\forall X=(x_1,...,x_n), Y=(y_1,...,y_n) \quad
d_{E^n}:=\frac{1}{n} \sum_{i=1}^{n}|x_i-y_i|^2,
\end{equation}
\end{definition}

Later we shall use the relation between $W_1$ and $W_2$ given in
the following proposition.

\begin{proposition}[\cite{HaMi}]\label{proposition2}
Given two symmetric probability measures $F^N$ and $G^N$ on $E^N$,
let us denote for any $k>0$
$$ {\mathcal M}_k:=M_k(F^N_1) + M_k(G^N_1)$$
where $M_k(f)$ denotes the k-th moment of the probability measure
$f$ on $E$.

One has $W_1(F^N,G^N)\leq W_2(F^N,G^N)$. For any $k > 2$, one has
\begin{equation}\label{Wrelation}
W_2(F^N,G^N)\leq 2^{\frac{3}{2}}{\mathcal M}_k^{1/k}
W_1(F^N,G^N)^{1/2-1/k}.
\end{equation}
\end{proposition}

Following \cite{HaMi}, we first introduce the \textit{normalized
versions} of the relative entropy and the Fisher information
between our measures.

\begin{definition}\label{definition8}({\bf (normalized) Relative
Entropy}) The normalized relative entropy between the measures
${G}_{N}$ and ${G}^{\otimes N}$ is given by
\begin{equation}\label{relativentropy}
H(\rho_{N}|\rho^{\otimes N}_{GP}):=\frac{1}{N}\int_{E^N} \rho_{N}
log(\frac{\rho_{N}}{\rho^{\otimes N}_{GP}})
\end{equation}
\end{definition}

\begin{remark} Differently from the entropy functional, the
relative entropy does not need any moment regularity assumption
because the integrand is a non negative function.
\end{remark}

\begin{definition}\label{definition9}({\bf Fisher
Information}) The Fisher information associated with a given
measure ${G}_{N}$ is given, when  ${G}_{N}\in W^{1,1}(E^N)$, by
\begin{equation}\label{Fischerinformation}
I(\rho_{N}):= \int_{E^N} \frac{|\nabla
\rho_{N}|^2}{\rho_N}=\int_{E^N} {|\nabla
\log{\rho_{N}}|^2}{\rho_N}
\end{equation}
\end{definition}
In our contest the hypothesis of finite energy condition (Section
2) implies that the Fisher Information is well-defined for all
$N$.
\begin{definition}\label{definition10}{\bf (normalized) Relative Fisher Information}
The normalized relative Fisher information between the measures
${G}_{N}$ and ${G}^{\otimes N}$ is given by
\begin{equation}\label{relativefischer}
I(\rho_{N}|\rho^{\otimes N}_{GP}):=\frac{1}{N} \int_{E^N} |\nabla
log\frac{\rho_{N}}{\rho^{\otimes N}_{GP}}|^2\rho_N
\end{equation}
\end{definition}

\begin{theorem}\label{theorem6}({\bf  ${G}_{N}$ is  $G$-entropy chaotic})

Under the hypothesis h1),h2),h3), and if the trapping potential
$V$ is such that $V({\bf r})\geq \alpha {\bf r}^{2+\epsilon}+
\beta$ with $\epsilon > 0$, where $\alpha$ and $\beta$ are
constants, the symmetric law ${G}_N$ is $G$-entropy
\textit{chaotic} according to Definition \ref{definition6}.
\end{theorem}
\begin{proof}
Following (\cite{HaMi}, Proposition 3.8), since $\rho_N$ and
$\rho_{GP}$ are probability densities in $\R^{3N}$ having finite
second moments, we can apply the HWI inequality by Otto-Villani
with respect to the Gaussian density
$g_{\lambda}(v)=\frac{1}{(2\pi
\lambda)^{3N/2}}\exp{(-|v|^2/2\lambda)}$ where $v=({\bf
r}_1,...,{\bf r}_N)$ (see \cite{OttoVillani}, Remarks after the
proof of Theorem 3):
\begin{equation}\label{WHI1}
\hat{H}(\rho_N|g_{\lambda})\leq
\hat{H}(\rho_{GP}|g_{\lambda})+W_2(\rho_N,\rho_{GP})\sqrt{\hat{I}(\rho_N|g_{\lambda})}
\end{equation}
Now
$$\hat{H}(\rho_N|g_{\lambda})=\hat{H}(\rho_N)-\int
\rho_Nlog(g_{\lambda})=\hat{H}(\rho_N)+\frac{3N}{2} log(2\pi
g_{\lambda})+\frac{M_2(\rho_N)}{2\lambda}$$ where $M_2$ denotes
the second moment, and, by definition of relative Fisher
information, we have
$$\hat{I}(\rho_N|g_{\lambda})=\int \rho_N|\nabla
log(\rho_N)+\frac{v}{\lambda}|^2=I(\rho_N)+\frac{2}{\lambda}\int
v\cdot \nabla \rho_N +\frac{M_2(\rho_N)}{\lambda^2}.$$ By
substituting these expressions into \eqref{WHI1}, simplifying the
terms containing  $log(g_{\lambda})$ and then sending $\lambda$ to
$+\infty$ and dividing the resulting limit by $N$ we obtain the
inequality:
$$H(\rho_N)\leq H(\rho^{\otimes N}_{GP})+W_2(\rho_N,\rho^{\otimes
N}_{GP})\sqrt{I(\rho_N)}.$$ By exchanging the two probability
measures $\rho_N$ and $\rho^{\otimes N}_{GP}$, we can recover in
the same way:
$$H(\rho^{\otimes N}_{GP})\leq H(\rho_N))+W_2(\rho_N,\rho^{\otimes
N}_{GP})\sqrt{I(\rho^{\otimes N}_{GP})}$$ and, finally, the
inequality
\begin{equation}\label{entropydistance}
|H(\rho_N)- H(\rho^{\otimes N}_{GP})|\leq W_2(\rho_N,\rho^{\otimes
N}_{GP})(\sqrt{I(\rho_N)}+ \sqrt{I(\rho^{\otimes N}_{GP})}).
\end{equation}
In order to take advantage of Proposition \ref{proposition2}, we
observe that we must require that the following moment
$$ {\mathcal M}_k:=M_k(\rho^{(1)}_N) + M_k(\rho_{GP})$$
is finite with $k=2+\epsilon$, for some $\epsilon > 0$. This
condition is implied by the condition on the confining potential
$V$ stated in the theorem.

By applying \eqref{Wrelation} to \eqref{entropydistance} we obtain
\begin{equation}\label{WHI2}
|H(\rho_N)- H(\rho^{\otimes N}_{GP})|\leq C
W_1(\rho_N,\rho^{\otimes N}_{GP})^{1/2-1/k}(\sqrt{I(\rho_N)}+
\sqrt{I(\rho^{\otimes N}_{GP})}).
\end{equation}
Since $I(\rho_N)$ is bounded for all $N$ and using the properties
of the Fisher information $I(\rho^{\otimes N}_{GP})=I(\rho_{GP})$
(\cite{HaMi}) and the fact that $I(\rho_{GP})$ is finite, and in
addition that the Kac's chaos holds (Theorem \ref{theorem4}) and
the fact that
 $\rho^{(1)}_N$ weakly converges to
$\rho_{GP}$ by Theorem \ref{theorem1}, we have that the entropy
chaos also holds (according to Definition \ref{definition6}).
\end{proof}

\begin{remark} Since under our assumptions $I(\rho_N)$ is bounded for
all $N$ and  $I(\rho_{GP})$ is finite, we note that by the
relevant inequality \eqref{entropydistance}  the entropy chaos
would follow from  the convergence to zero of the MKW distance of
order $2$ between the  densities $\rho_N$ and $\rho^{\otimes
N}_{GP}$. However, to prove this type of convergence between
probability measures is in general not a trivial problem (see,
e.g. \cite{Villani}, Chapter 6).
\end{remark}

We also establish a useful HWI type inequality  for the relative
entropy between ${G}_{N}$ and ${G}^{\otimes N}_{GP}$.

\begin{theorem}\label{theorem7}({\bf HWI inequality})
Under the hypothesis h1),h2), h3), and if the confining potential
$V$ is convex and $C^{\infty}$ one has
\begin{equation}\label{HWI}
H(\rho_N,\rho^{\otimes N}_{GP})\leq
\frac{1}{\sqrt{N}}W_2(\rho_N,\rho^{\otimes
N}_{GP})\sqrt{I(\rho_N,\rho^{\otimes N}_{GP})}
\end{equation}
\end{theorem}

\begin{proof}
Let us introduce $W$ such that $d{G}^N= \exp{(-W)}d{\bf r}_1\cdot
\cdot \cdot d{\bf r}_N$ on $E^N$, i.e. $W({\bf r}_1,...,{\bf
r}_N)= -\sum_{i=1}^N log(\rho_{GP}({\bf r}_i))$. Since $V\in
C^{\infty}(E)$ we have that $\phi_{GP} \in C^{\infty}$ (see
\cite{Lieb2}), and, therefore, $W \in C^2(E^N)$. We show now that
$W$ is convex. From \cite{Lieb} we know that $\phi_{GP}$ is
log-concave. We report the  proof of this fact for completeness.
Let us consider the GP equation \eqref{GP}. If one is able to
prove that the solutions of the equations
\begin{equation}\label{logconcave}
\partial_tu - \nabla^2u=0,\quad \partial_tu+Vu=0, \quad
\partial_tu+8\pi u^2=\lambda u
\end{equation}
are log-concave when u(0,{\bf x}) is positive and log-concave,
then applying Trotter formula one has that the solution of
equation \eqref{GP} is log-concave. Since the convolution of two
log-concave functions is log-concave, the solution of the first
equation in \eqref{logconcave} is log-concave. The solution of the
second equation in \eqref{logconcave} is log-concave due to the
convexity of $V$. The log-concavity of the solution of the third
equation in \eqref{logconcave} is proved in (\cite{Lions}, Theorem
1). Since $W({\bf r}_1,...,{\bf r}_N)= -\sum_{i=1}^N
log(\rho_{GP}({\bf r}_i))$ and the sum of log-concave functions is
also log-concave, it follows that $W$ is convex in $E^N$.

Since $\rho_Nd{\bf r}_1\cdot \cdot \cdot d{\bf r}_N$ is absolutely
continuous with respect to $ \rho_{GP}d{\bf r}_1\cdot \cdot \cdot
d{\bf r}_N $, by (\cite{OttoVillani}, Theorem 3), we have that
\begin{equation}
\hat{H}(\rho_N,\rho^{\otimes N}_{GP})\leq W_2(\rho_N,\rho^{\otimes
N}_{GP})\sqrt{\hat{I}(\rho_N,\rho^{\otimes N}_{GP})}
\end{equation}
where $\hat{H}$ respectively $\hat{I}$ are the non-normalized
relative entropy and respectively relative Fisher information.
Dividing by N we finally obtain \eqref{HWI}.
\end{proof}

\begin{remark}
We can also prove using Theorem \ref{theorem7} that entropy chaos
holds. In fact under the assumptions on the confining potential
$V$ in Theorem \ref{theorem4} plus the assumption of convexity of
$V$, it is sufficient to substitute the inequality in Proposition
\ref{proposition2} into the HWI inequality \eqref{HWI} for
obtaining that $H(\rho_N,\rho^{\otimes N}_{GP})$ converges to zero
when $N$ goes to infinity.
\end{remark}

We have seen that the Kac's chaos property is linked with the weak
convergence of the marginal densities. Since under certain
regularity conditions on the trapping potential $V$ we have both
Kac's chaos and entropy chaos, it is natural to investigate
whether the entropy convergence implies a stronger convergence of
the marginal densities. The answer is affirmative.

Using some well-known results concerning the convergence of the
entropies and the strong $L^1-$ convergence of the probability
densities (see, e.g., \cite{BoLe}), in our case we obtain the
following.

\begin{theorem}\label{theorem8}
Let ${G}^N_1$ and $G$ as above. If the entropy-chaos is satisfied
(in the sense of Definition \ref{definition6}), then ${G}^N_1$
converges to $G$ in total variation.
\end{theorem}

\begin{proof}
We first prove that from the convergence of the entropies in
\eqref{entropychaos} it follows that

\begin{equation}\label{entropyconvergence}
H({\rho}^{(1)}_{N}|\gamma) \rightarrow H(\rho_{GP}|\gamma) <
+\infty
\end{equation}
where $\gamma({\bf r})=\frac{1}{(2\pi \lambda)^{3N/2}}\exp{(-|{\bf
r}|^2/2\lambda)}, {\bf r}\in \R^3, \lambda>0$. In fact
\begin{equation}
H({\rho}^{(1)}_{N})=H({\rho}^{(1)}_{N}|\gamma) +\int
{\rho}^{(1)}_{N}\log{\gamma}
\end{equation}
and for $N\uparrow \infty$
\begin{equation}
\int {\rho}^{(1)}_{N}\log{\gamma}\rightarrow \int
{\rho}_{GP}\log{\gamma}
\end{equation}
By the convergence of the entropies from this
\eqref{entropyconvergence} follows. Now from a well-known result (
see, e.g. \cite{BoLe}, Lemma 2.5), which is true only on finite
measure spaces ( see \cite{DeVecchi} for some counterexamples) ,
when \eqref{entropyconvergence} holds and moreover
${\rho}^{(1)}_{N}d{\bf r} \rightharpoonup \rho_{GP}d{\bf r}$, then
as ${N\uparrow \infty}$
\begin{equation}
\int |\frac{{\rho}^{(1)}_{N}}{\gamma} -
\frac{\rho_{GP}}{\gamma}|\gamma d{\bf r}\rightarrow 0
\end{equation}
From the latter result it follows that ${\rho}^{(1)}_{N}$
converges to $\rho_{GP}$ strongly in $L^1(d{\bf r})$.

Since by using Scheffe's theorem, one can easily prove that in
$\R^d$ the total variation distance between  two absolutely
continuous measures is equal to $\frac{1}{2}$ the $L^1-$ distance
between the two probability densities \footnote{In fact if
${\mathbb{Q}}_1$ and $\mathbb{Q}_2$ are two absolutely continuous
measures  in $\R^d$ with densities $f_1$ and $f_2$ respectively ,
then $d_{VT}(\mathbb{Q}_1,\mathbb{Q}_2):=\sup_{A\in {\mathcal
B}(\R^d)}|{{\mathbb{Q}}}_1(A)-{\mathbb{Q}}_2(A)|=\sup_{A\in
{\mathcal B}(\R^d)}|\int_{A}f_1({\bf r})d{\bf r}-\int_{A}f_2({\bf
r})d{\bf r}|=\int_{f_1>f_2} (f_1({\bf r})-f_2({\bf r}))d{\bf
r}=\int_{f_2>f_1} (f_2({\bf r})-f_1({\bf r})){\bf
r}=\frac{1}{2}\int|f_1({\bf r})-f_2({\bf r})|d{\bf r}$} (see, e.g.
\cite{DeLu}), we can state that our one-particle marginal
probability measure $G^N_1$ converges to $G$ in the sense of total
variation convergence of probability measures.
\end{proof}

\section{Existence of a weakly convergent subsequence}

In the previous Section 4 and Section 5 we have presented some
convergence results for the fixed time marginal density
$\rho_{N}$.

In the present section we focus on a convergence problem for the
one particle probability measure ${\mathbb P}^N_1$ on the path
space. First the results contained in \cite{MU},\cite{MU1} and
\cite{DeVU} are briefly recalled in order to properly investigate
the asymptotic behavior of the {\it one particle relative
entropy}.

We consider the measurable space $(\Omega^N,\mathcal F ^N)$  where
$\Omega$ is $C(\R^+ \rightarrow \R^{3})$, $N\in {\mathbb N}$ and
$\mathcal F $ is its Borel sigma-algebra as introduced in Section
2. We denote by $\hat Y := (Y_1,\dots,Y_N)$ the coordinate process
and by $\mathcal F ^N_t$ the natural filtration.

Let us  introduce a process $X^{GP}$ with invariant density
$\rho_{GP}$ and try to compare it with the generic interacting
one-particle \textit{non} markovian diffusion $Y_1(t)$.

We assume that $X^{GP}$ is a weak solution of the SDE
\begin{equation} \label{SDE}
dX^{GP}_t:= u_{GP}(X^{GP}_t)dt + (\frac {\hbar}{m})^{\frac 1
2}dW_t
\end{equation}

\noindent where,
$$
u_{GP} := \frac {1}{2}\frac {\nabla \rho_{GP}}{\rho_{GP}}
$$

We denote again by ${\mathbb P}_N$ respectively ${\mathbb
P}^N_{GP}$ the measures corresponding to the weak solutions of the
$3N$- dimensional stochastic differential equations

\begin{equation} \label {SDEa}
\hat Y_t-\hat Y_0 = \int_0^t \hat b ^N (\hat Y_s)ds +\hat W_t
\end{equation}
respectively
\begin{equation} \label {SDEb}
\hat Y_t-\hat Y_0 = \int_0^t \hat u_{GP} (\hat Y_s)ds +\hat W_t',
\end{equation}

\noindent where
$$
\hat u_{GP}({\bf r}_1,\cdots,{\bf r}_N)=(u_{GP}({\bf
r}_1),\cdots,u_{GP}({\bf r}_N)),
$$
\noindent $ \hat Y_0$  is a random variable with probability
density equal to $ \rho_N$, while  $\hat W_t$ and $\hat W_t'$ are
$3N$-dimensional $\mathbb P_N$ and $\mathbb P^N_{GP}$ standard
Brownian motions, respectively.

In this section we use the shorthand notation $ \hat {b}^N_s=:\hat
{b}^N(\hat {Y}_s) $ and $ \hat {u}^N_s=:\hat {u}_{GP}(\hat {Y}_s)
$.

Following \cite{MU} we compute the relative entropy between the
three-dimensional {\it one-particle} non markovian diffusion $Y_1$
and $X^{GP}$.

In order to use Girsanov Theorem, we will assume that $u_{GP}$ is
bounded. Then the following finite energy conditions hold:

\begin{equation}\label {eni}
 E_{\mathbb P_N} \int_0^t \|\hat b ^N_s \|  ^2 ds  < \infty
\end{equation}

 \begin{equation}\label{enii}
 E_{\mathbb P_N} \int_0^t\| \ \hat u^{GP}_s  \| ^2 ds   < \infty ,
 \end{equation}

 \noindent which follow from the fact that $\Psi^0_N$ is the minimizer of
$E^N[\Psi]$, and our hypothesis on $u_{GP}$. It is well known that
these are also \textit{finite entropy conditions} (see, e.g.
\cite{Follmer}) which imply that $\forall t\geq 0$

$$\mathbb P_N|_{\mathcal F_t} \ll \hat W|_{\mathcal F_t}, \quad \mathbb P^N_{GP}|_{\mathcal F_t}\ll \hat W'|_{\mathcal F_t}  $$
(where $\ll$ stands for absolute continuity)
 Then, by Girsanov's theorem, we have, for all $t>0$,
\begin{equation}\label{derivative}
 \frac {d\mathbb P_N}{d\mathbb P^N_{GP}}|_{\mathcal F_t}=
  \exp \{-\int_0^{t}  (\hat b ^N _s-
 \hat u^{GP}_s)\cdot
 d \hat W_s+\frac{1}{2} \int_0^ {t}  \|\hat b ^N_s - \hat u^{GP}_s\|^2ds
\},
 \end{equation}

\noindent  where $|.|$ denotes the Euclidean norm in $\R^{3N}$.
The relative entropy restricted to $\mathcal F_t$ reads

 \begin{multline} \label{eq: Entropy}
 \mathcal H(\mathbb P_N,\mathbb P^N_{GP}) |_{\mathcal F_{t}}
 =: \mathbb E_{\mathbb P_N}[
 \log \frac {d\mathbb P_N}{d\mathbb P^N_{GP}}\mid_{\mathcal F_t}]=
 \frac {1}{2}E_{\mathbb P_N}\int_0^{t}  \|\hat b ^N_s - \hat u^{GP}_s\|^2 ds
 \end{multline}

  Since under $\mathbb P_N$ the $3N$-dimensional process $\hat Y$ is a solution of \eqref {SDEa}
  with invariant probability density $ \rho_N$ , we can write, recalling also \eqref{eni} and \eqref{enii},

 \begin{multline}
 \frac {1}{2}E_{\mathbb P^N}\int_0^{t}  \|\hat {b} ^N_s - \hat {u}^{GP}_s\|^2 ds=\\
 =\frac {1}{2}\int_0^{t}E_{\mathbb P_N}  \|\hat {b} ^N_s - \hat{u}^{GP}_s\|^2 ds=\\
 =\frac 1 2 t \int _{\R^{3N}} \|\hat {b} ^N({\bf r}_1,\dots,{\bf r}_N) -
  \hat {u}_{GP}({\bf r}_1,\dots,{\bf r}_N)\|^2  \rho_N d{\bf r}_1\dots d{\bf r}_N
 \end{multline}

 \noindent so that we get

 \begin{multline}
  \mathcal H(\mathbb P_N,\mathbb P^N_{GP}) |_{\mathcal F_{t}} =\\
 = \frac 1 2 t
  \int _{\R^{3N}} \sum _{i=1}^{N}\| b ^N_i({\bf r}_1,\dots,{\bf r}_N) -  u_{GP} ({\bf r}_i)\|^2 \rho_N d {\bf r}_1 \dots d {\bf r}_N =\\
  =
 \frac 1 2 N t\int _{\R^{3N}} \| b ^N_1({\bf r}_1,\dots,{\bf r}_N) -  u_{GP} ({\bf r}_1)\|^2 \rho_N d {\bf r}_1 \dots d {\bf r}_N =\\
 =\frac 1 2 N E_{\mathbb P_N}\int_0^{t}  \| b ^N_1 (\hat Y_s) -  u_{GP}(Y_1(s))\|^2 d
s,
 \end{multline}
 \noindent where the symmetry of $\hat b^N $ and $\rho_N$ has been exploited.

 Finally we get the sum of $N$ identical one-particle relative entropies, each of them being defined by

 \begin{multline}
 \bar{\mathcal H} (\mathbb P_N, \mathbb P^N_{GP})|_{\mathcal F_t}=:
 \frac 1 N  \mathcal H(\mathbb P_N,\mathbb P^N_{GP}) |_{\mathcal F_{t}} = \\
 =\frac 1 2 E_{\mathbb P_N}\int_0^{t}  \| b ^N_1(\hat Y_s) -  u^{GP}(Y_1(s))\|^2 d s
 \end{multline}

By  Theorem \ref{theorem2} we deduce that for any $ t
> 0$ the one particle relative entropy does not go to zero in the scaling
limit but it is asymptotically \textit{finite}.

In particular from \eqref{E1} we obtain that:

\begin{multline}\label{limitentropy}
 lim_{N\uparrow +\infty} \bar{\mathcal H} (\mathbb P_N, \mathbb P^N_{GP})|_{\mathcal F_t}= \\
 =lim_{N\uparrow +\infty} \frac 1 2 E_{\mathbb P_N}\int_0^{t}  \| b ^N_1(\hat Y_s) -  u^{GP}(Y_1(s))\|^2 d
s=\int_0^t g\hat{s}\int_{\R^3}|\rho_{GP}|^2d{\bf r}
\end{multline}
\noindent where $\hat{s}\in (0,1]$ is a constant depending on the
interaction potential $v$ through the solution of the zero-energy
scattering equation (see Theorem \ref{theorem2}). For every $t \in
[0,T]$ with an arbitrary finite $T$, the right and side of
\eqref{limitentropy} is a {\it finite constant}.

\vskip10pt

We recall a result, proved in \cite{DeVU}, which extends to our
case a useful chain-rule for the relative entropy.

\begin{lemma}\label{lemma1} We consider $M=X\times Y$, where $X$ and $Y$ are
Polish spaces. Let ${\mathbb P}$ be a measure on $M$ and ${\mathbb
Q}_1$ and ${\mathbb Q}_2$ measures on $X$ and $Y$ respectively. We
denote by $\mathbb Q={\mathbb Q}_1\otimes {\mathbb Q}_2$ the
product measure on $M$ of the measures ${\mathbb Q}_1$ and
${\mathbb Q}_2$ and we suppose that ${\mathbb P} \ll {\mathbb Q}$.
Then we have
\begin{equation}
{\mathcal H}({\mathbb P}|{\mathbb Q})\geq {\mathcal H}({\mathbb
P}_1|{\mathbb Q}_1) + {\mathcal H}({\mathbb P}_2|{\mathbb Q}_2),
\end{equation}
\noindent where ${\mathbb P}_1$ and ${\mathbb P}_2$ are the
marginal probabilities of ${\mathbb P}$.
\end{lemma}

\vskip20pt

We now consider the one-particle marginal ${\mathbb P}^1_N$ of the
symmetric measure ${\mathbb P}_N$, defined on $\Omega^N$.

The following theorem holds (\cite{DeVU}).

\begin{theorem}[Marginal Entropy Estimate]\label{theorem9} For the
relative entropy of the one-particle marginal ${\mathbb P}^1_N$
versus the measure ${\mathbb P}_{GP}$ one has:
\begin{equation}
{\mathcal H}({\mathbb P}^1_N|{\mathbb P}_{GP})\leq \frac{1}{N}
{\mathcal H}({\mathbb P}_N|{\mathbb P}^{N}_{GP})
\end{equation}
\end{theorem}

\begin{proof}
We can use Lemma \ref{lemma1}. In fact if we decompose
$\Omega^N=\Omega \times \Omega^{N-1}$ the hypothesis in Lemma
\ref{lemma1} are satisfied and we obtain:
\begin{equation}
{\mathcal H}({\mathbb P}_N|{\mathbb P}^N_{GP})\geq {\mathcal
H}({\mathbb P}^1_N|{\mathbb P}_{GP}) + {\mathcal H}({\mathbb
P}^{N-1}_N|{\mathbb P}^{N-1}_{GP}),
\end{equation}
where with ${\mathbb P}^{N-1}_N$ we denote the marginal of
${\mathbb P}_N$ with respect to $N-1$ particles and, of course,
for the symmetry of ${\mathbb P}_N$ it does not matter which
particles we take. Applying again Lemma \ref{lemma1} we have:
\begin{equation}
{\mathcal H}({\mathbb P}^{N-1}_N|{\mathbb P}^{N-1}_{GP})\geq
{\mathcal H}({\mathbb P}^{1}_N|{\mathbb P}_{GP}) + {\mathcal
H}({\mathbb P}^{N-2}_N|{\mathbb P}^{N-2}_{GP}),
\end{equation}
So by taking $N-2$ consecutive applications of Lemma \ref{lemma1}
we obtain the stated result.
\end{proof}

\vskip20pt

The next theorem, proved in \cite{DeVU}, uses mainly the fact that
the relative entropy has the property of the compactness of level
sets.

\vskip10pt

\begin{theorem}[Existence Theorem]\label{theorem10} On the space
$(\Omega,\mathcal F, {\mathbb P}_{GP})$ there exists a probability
measure $\hat {\mathbb P}$ such that:

i) ${\mathbb P}^1_{N_j}$ weakly converges to $\hat {\mathbb P}$
for some subsequence ${\mathbb P}^1_{N_j}$ of ${\mathbb P}^1_{N}$;

ii) $\hat {\mathbb P}$ is absolutely continuous with respect to
${\mathbb P}_{GP}$.
\end{theorem}

\section{Uniqueness of the limit probability measure}

In this section we prove that $\hat{\mathbb{P}}$ is unique, that
is, not only a subsequence of $\mathbb{P}^1_N$ converges to
$\hat{\mathbb{P}}$, but the entire sequence converges to the same
limit probability measure.
In particular we show that $\hat{\mathbb{P}}=\mathbb{P}_{GP}$.\\
In order to do this we need some results obtained in \cite{MU}.

We consider the following time dependent random subset of $\R^3$

\begin {equation} \label{eq: D}
D_N(t) := \bigcup _{i=2}^N  B ^N(Y_i(t))
\end {equation}

\noindent where $B ^N({\bf r})$  is again the ball with radius
$N^{-\frac{1}{3}-\frac{1}{\delta}} $ centered in ${\bf r}$ (see
Lemma 7.3 in \cite{Lieb2}), and the stopping time

\begin {equation} \label{eq: Tau}
\tau ^N:= inf \{ t \geq 0: Y_1(t) \in D_N(t)         \}
\end {equation}

\subsection{The stopped measures}
We investigate some properties of the measure obtained by stopping
$\mathbb{P}_N$ with respect to the stopping time $\tau_N$.
We start by recalling two relevant properties of $\tau_N$.\\
The first is that for all $t>0$ (\cite{MU},Proposition 2)
\begin{equation}\label{stoppingtime}
\lim_{N \rightarrow + \infty}\mathbb{P}_N(\tau_N \geq t)=1,
\end{equation}
\\
It is important to stress that $\tau_N$ strongly depends on the
measure ${\mathbb P}_N$. We consider, for a fixed value of $N$ and
for the one-particle diffusion $Y_1$, the random region $D_N$ in
$\R^3$ consisting of the union of $N-1$ balls, each localized
where one of the other $N-1$ particles is at time $t$. The joint
distribution of the $N$ particles is described by the probability
law ${\mathbb P}_N$.

The second property is as follows. Let us define the stopped
measures
$$\tilde{\mathbb{P}}_N=\mathbb{P}_N|_{\mathcal{F}_{\tau_N}},$$
where $\mathcal{F}_{\tau_{N}}$ is the sigma-algebra associated
with the stopping time $\tau_N$ given by
$$ \mathcal{F}_{\tau_{N}} =\{A \in {\mathcal{F}},A\cap\{\tau_N \leq t\}\in \mathcal{F}_{t}\}     $$

We note that $\tilde{\mathbb{P}}_N$ is absolutely continuous with
respect to $\mathbb{P}_{GP}^N$ and moreover the following relation
holds:
$$\frac{d\tilde{\mathbb{P}}_N}{d\mathbb{P}_{GP}^N}=E_{\mathbb{P}^N_{GP}}\left[\left.\frac{d\mathbb{P}_N}
{d\mathbb{P}_{GP}^N}\right|\mathcal{F}_{\tau_N}\right].$$ In this
way we can look at $\tilde{\mathbb{P}}_N$ as a measure defined on
the whole sigma algebra $\mathcal{F}_T$ having, by definition, the
above
density.\\
Since $\tau_N$ represents the stopping time before the
one-particle process $Y_1$ enters in the \textit{interaction
random set} $D_N(t)$ generated by the other particles, one can
prove that
$$\lim_{N \rightarrow + \infty}\frac{1}{N}H(\tilde{\mathbb{P}}_N||\mathbb{P}_{GP}^N)=0$$
(see \cite{MU}, Proposition3).\\
Let us now consider $\tilde{\mathbb{P}}_N^1$ as the measure
$\tilde{\mathbb{P}}_N$ projected into the sigma-algebra
$\mathcal{F}^1 $ generated by the first particle, that is the
measure such that
$$\frac{d\tilde{\mathbb{P}}^1_N}{d\mathbb{P}_{GP}}=E_{\mathbb{P}_{GP}^N}\left[\left. \frac{ d\tilde{\mathbb{P}}_N}{d\mathbb{P}_{GP}^N}
\right|\mathcal{F}^1 \right]. $$ Of course in general
$\tilde{\mathbb{P}}_N^1\not = \mathbb{P}_N^1$ and there is no
direct connection between the two measures. While
$\tilde{\mathbb{P}}_N$ has a clear physical meaning being the
measure corresponding to the $N$ particles process stopped
according to $\tau_N$, $\tilde{\mathbb{P}}_N^1$ has no physical
meaning. In fact when we focus our attention only on the single
particle, we cannot know when it will interact with the other
particles. This is only possible when we consider all the
particles. From the mathematical point of view, this physical fact
corresponds to the non $\mathcal{F}^1-$ measurability of the
stopping time
$\tau_N$.\\
By a reasoning similar to the one in Theorem \ref{theorem9} we can
prove that
$$H(\tilde{\mathbb{P}}_N^1||\mathbb{P}_{GP})\leq \frac{1}{N}H(\tilde{\mathbb{P}}_N||\mathbb{P}_{GP}^N)\rightarrow 0.$$
By the well-known Csiszar-Kullback inequality
(\cite{Csiszar},\cite{Kullback}), we deduce from this that the
sequence $\tilde{\mathbb{P}}_N^1$ converges in total variation to
$\mathbb{P}_{GP}$.

We have seen in Section 5 that in $\R^d$  the total variation
distance between two absolutely continuous measures is equal to
$\frac{1}{2}$ times the $L^1-$ distance between the corresponding
two probability densities . In our path space the following result
suffices.

\begin{proposition}
Given the two absolutely continuous measures
$\tilde{\mathbb{P}}^1_N$ and ${\mathbb P}_{GP}$ we have that
\begin{equation}\label{distvariation}
d_{TV}(\tilde{\mathbb{P}}^1_N,{\mathbb{P}}_{GP})\geq
\frac{1}{2}\int|\frac{d\tilde{\mathbb{P}}^1_N}{d{\mathbb{P}}_{GP}}-1|{d{\mathbb{P}}_{GP}}
\end{equation}
\end{proposition}

\begin{proof}
We recall the definition
\begin{equation}
d_{TV}(\tilde{{\mathbb P}}^1_N,{\mathbb{P}}_{GP}):=\sup_{A\in
{\mathcal
F}^1_T}|\tilde{{\mathbb{P}}}^1_N(A)-{\mathbb{P}}_{GP}(A)|\\
=\sup_{A\in {\mathcal F}^1_T}|\int_A
(\frac{d\tilde{\mathbb{P}}^1_N}{d\mathbb{P}_{GP}}-1){d\mathbb{P}_{GP}}|
\end{equation}
By taking
$$A=\{(\frac{d\tilde{\mathbb{P}}^1_N}{d\mathbb{P}_{GP}}-1)\geq 0\}$$
we have that
\begin{equation}
d_{TV}(\tilde{\mathbb{P}}^1_N,{\mathbb{P}}_{GP})\geq
\int_A(\frac{d\tilde{\mathbb{P}}^1_N}{d\mathbb{P}_{GP}}-1){d\mathbb{P}_{GP}}
\end{equation}
and
\begin{equation}
d_{TV}(\tilde{\mathbb{P}}^1_N,{\mathbb{P}}_{GP})\geq
-\int_{A^c}(\frac{d\tilde{\mathbb{P}}^1_N}{d\mathbb{P}_{GP}}-1)d\mathbb{P}_{GP}
\end{equation}
As a consequence
\begin{equation}
\int|\frac{d\tilde{\mathbb{P}}^1_N}{d\mathbb{P}_{GP}}-1|{d\mathbb{P}_{GP}}\leq
2d_{TV}(\tilde{\mathbb P}^1_N,{\mathbb{P}}_{GP})
\end{equation}

\end{proof}
Since $\tilde{\mathbb{P}}_N^1$ converges in total variation to
${\mathbb{P}}_{GP}$,  by the previous theorem  we obtain that
$\frac{d\tilde{\mathbb{P}}^1_N}{d{\mathbb{P}}_{GP}}$ converges to
$1$ strongly in $L^1(\mathcal{F}^1_T,\mathbb{P}_{GP})$.

\subsection{The unique weak limit}

Let us now consider a subsequence of $\mathbb{P}_N^1$ weakly
converging to $\hat{\mathbb{P}}$. For simplicity we denote this
subsequence again by $\mathbb{P}_N^1$.\\
Let us recall that, the weak limit $\hat{\mathbb{P}}$ being
absolutely continuous with respect to $\mathbb{P}_{GP}$, for every
$A \in \mathcal{F}^1_T, T>0,$ such that $\hat{\mathbb{P}}(\partial
A)=\mathbb{P}_{GP}(\partial A) =0$ one has by definition of weak
convergence of measures that
$$\lim_{N \rightarrow +\infty}E_{\mathbb{P}^1_N}[I_A]=E_{\hat{\mathbb{P}}}[I_A].$$
\\
Let us put
$$B_N=\{\tau_N \geq T\}.$$
We recall that $A \cap B_N \in \mathcal{F}_{\tau_N}$. In fact for
every $t\in [0,T)$ we have that
$$A \cap B_N \cap \{\tau \leq t\}=\emptyset \in \mathcal{F}_t.$$
Obviously
$$I_A=I_{A \cap B_N}+I_{A \cap B_N^C},$$
where $B_N^C$ is the complement set of $B_N$. By the properties of
the characteristic function $I_K$ of any set $K$ we have that
$$I_{A \cap B_N} \leq I_A \leq I_{A \cap B_N} + I_{B_N^C},$$
from which we obtain that
$$E_{\mathbb{P}_N}[I_{A \cap B_N}] \leq E_{\mathbb{P}_N}[I_A] \leq E_{\mathbb{P_N}}[I_{A \cap B_N}]+\mathbb{P}_N({B_N^C}).$$
From the latter relation and the fact that  $\mathbb{P}_N({B_N^C})
\rightarrow 0$ as $N \rightarrow +\infty$, for every fixed
$\epsilon>0$, we get for $N$ sufficiently big
$$E_{\mathbb{P}_N}[I_{A \cap B_N}] \leq E_{\mathbb{P}_N}[I_A] \leq E_{\mathbb{P_N}}[I_{A \cap B_N}]+\epsilon.$$
By the last relation and
$E_{\mathbb{P}_N}[I_A]=E_{\mathbb{P}_N^1}[I_A]$ if we send $N$ to
infinity and set $f_+=\limsup E_{\mathbb{P}_N}[I_{A \cap B_N}]$
and  $f_-=\liminf E_{\mathbb{P}_N}[I_{A \cap B_N}]$, we obtain
$$f_+\leq E_{\hat{\mathbb{P}}}[I_A] \leq f_-+\epsilon.$$
Finally, $\epsilon>0$ being arbitrary we have that
$$f=f_+=f_-=E_{\hat{\mathbb{P}}}[I_A].$$
We now study what happens to $E_{\mathbb{P}_N}[I_{A \cap B_N}]$ as $N\rightarrow +\infty$.\\
First of all, since $A \cap B_N \in \mathcal{F}_{\tau_N}$ we have
that
$$E_{\mathbb{P}_N}[I_{A \cap B_N}]=E_{\tilde{\mathbb{P}}_N}[I_{A \cap B_N}].$$
Recalling that $I_{A \cap B_N}=I_A I_{B_N}$, consequently we
obtain
\begin{eqnarray*}
E_{\tilde{\mathbb{P}}_N}[I_{A \cap B_N}]&=&E_{\tilde{\mathbb{P}}_N}[I_A I_{B_N}]\\
&=&E_{\tilde{\mathbb{P}}_N^1}[E_{\tilde{\mathbb{P}}_N}[I_A I_{B_N}|\mathcal{F}^1_T]]\\
&=&E_{\tilde{\mathbb{P}}_N^1}[I_AE_{\tilde{\mathbb{P}}_N}[I_{B_N}|\mathcal{F}^1_T]]\\
&=&E_{\mathbb{P}_{GP}}\left[I_A\frac{d\tilde{\mathbb{P}}_N^1}{d\mathbb{P}_{GP}}E_{\tilde{\mathbb{P}}_N}[I_{B_N}|\mathcal{F}^1_T]\right].
\end{eqnarray*}
Let us set
$$H_N:=E_{\tilde{\mathbb{P}}_N}[I_{B_N}|\mathcal{F}^1].$$
By the properties of the conditional expectation and of the
characteristic function we have that  $0 \leq H_N \leq 1$.
Therefore
$H_N$ is a sequence which is bounded in norm in $L^{\infty}(\mathcal{F}^1_T, \mathbb{P}_{GP})$.\\
Since $L^{1}(\mathcal{F}^1_T, \mathbb{P}_{GP})$ is  separable,
 ${\mathcal F}^1_T $ being a sub-$\sigma$ algebra of the
essentially separable $\sigma$ algebra ${\mathcal F}_T $,  then by
considering the dual pair $\langle L^1 , L^{\infty}\rangle$),
there exists a subsequence $H_{N_j}$ which is weak* convergent to
$\tilde{H}$ in $L^{\infty}$ by a well-known result on weak*
convergence.\\
We have previously proved that
$\frac{d\tilde{\mathbb{P}}_N^1}{d\mathbb{P}_{GP}} \rightarrow 1$
strongly in $L^1$, and so, $I_A$ being a bounded function, we have
that
$$I_A \frac{d\tilde{\mathbb{P}}_N^1}{d\mathbb{P}_{GP}} \rightarrow I_A,$$
strongly in $L^1$. \\
By the properties of weak and strong convergence we obtain that
$$E_{\mathbb{P}_{GP}}\left[I_A\frac{d\tilde{\mathbb{P}}_{N_j}^1}{d\mathbb{P}_{GP}}H_{N_j}\right] \rightarrow E_{\mathbb{P}_{GP}}[I_A \tilde{H}],$$
i.e.
\begin{equation}\label{equality}
E_{\mathbb{P}_{GP}}[I_A \tilde{H}]=f=E_{\hat{\mathbb{P}}}[I_A].
\end{equation}

Since this happens for an arbitrary set $A \in \mathcal{F}^1_T$
which is a $\mathbb{P}_{GP}-$ continuity set (i.e. $
{\mathbb{P}_{GP}}(\partial A)=0$) then the equality
\eqref{equality} is true for any set in $\mathcal{F}^1_T$. In
fact, let us consider the class $\mathcal{C}_b$ of the finite
intersections of open balls whose boundaries have null
${\mathbb{P}_{GP}}$ measure. This class is a $\pi$ system (i.e. it
is closed with respect to finite intersections) because
$$ \partial (B\cap C) \subseteq (\partial B) \cup (\partial C)
$$
The equality \eqref{equality} is evidently true for any set in
$\mathcal{C}_b$. Since $\Omega$ is separable, then
$\sigma(\mathcal{C}_b)= \mathcal{F}^1_T$, i.e. the $\mathcal{C}_b$
generates the Borel $\sigma-$ algebra $\mathcal{F}^1_T$. In fact
we know that by separability the finite intersections of open
balls generates $\mathcal{F}^1_T$ because each open set is a
countable union of open balls. On the other hands, one can prove
that the open balls having boundary with null ${\mathbb{P}_{GP}}$
measure are dense in the following sense. Given an arbitrary open
ball $B(x,\epsilon), \epsilon>0$, there exists an $r\in
(0,\epsilon)$ such that the boundary of $B(x,r)$ has
${\mathbb{P}_{GP}}$ measure zero because for different $r$ these
boundaries have empty intersection. Therefore we have that
$$\frac{d\hat{\mathbb{P}}}{d\mathbb{P}_{GP}}=\tilde{H}.$$
From this fact we deduce that not only the subsequence $H_{N_j}$
converges to $\tilde{H}$, but also the entire sequence $H_N$
converges to $\tilde{H}$ weakly* in $L^{\infty}$. Moreover since
$0\leq H_N \leq 1$ by the properties of weak* convergence we must
have $0\leq\tilde{H}\leq 1$.

Of course from $E_{\mathbb{P}_{GP}}[\tilde{H}]=1$, it necessarily
follows that $\tilde{H}=1$. This implies that
$$\hat{\mathbb{P}}=\mathbb{P}_{GP}.$$

Since every weakly convergent subsequence of $\mathbb{P}^1_N$
converges to $\mathbb{P}_{GP}$, we can deduce that the entire
sequence  $\mathbb{P}^1_N$ weakly converges to $\mathbb{P}_{GP}$.\\

We have proved the following:

\begin{theorem} (A weak convergence result)\\
Under the hypothesis h1), h2), and h3), the one-particle marginal
measure $\mathbb{P}^1_N$ weakly converges to $\mathbb{P}_{GP}$,
the weak solution to \eqref{SDE}, where $\phi_{GP}$ is the unique
minimizer of the GP functional \eqref{GPfunc}.
\end{theorem}

\begin{remark}
For the proof of our convergence result we did not have at
disposal the well-known Lemma 11.1.1 by Stroock-Varadhan
(\cite{StroockVaradhan}). The reason for this is that we can not
decouple the sequence of stopping times $\tau_N$ from the sequence
of probability measures ${\mathbb P}_N$ because they are
intrinsically coupled, in our case.
\end{remark}

\vskip10pt \noindent {\bf Acknowledgment.} The authors have
benefited from the gracious hospitality of the Centre
Interfacultaire Bernoulli at the EPFL, Lausanne, during the
semester: Geometric Mechanics, Variational and Stochastic Methods
(1 January, 30 June 2015) organized by S. Albeverio, A.B. Cruzeiro
and D. Holm. The second and third authors also thanks the NSF for
financial support.


\begin{thebibliography}{}

\bibitem{Adami} Adami R., Golse  F. ,Teta A.: Rigorous derivation of the cubic NLS in dimension one,
{\em Journal of Statistical Physics} {\bf 127} (2007), no. 6,
1193--1220.



\bibitem{AU} Albeverio, S., Ugolini, S.: A Doob h-transform of the
Gross-Pitaevskii Hamiltonian. {\em Journal of Statistical Physics}
(in print) (2015)

\bibitem{Billingsley} Billingsley P.: {\em Convergence of probability
measures}, second edition, John Wiley and Sons, New York, 1999.

\bibitem{BoLe} Borwein J.M., Lewis A.S.: Convergence of best
entropy estimates. {\em SIAM J. Optimization} {\bf 1},{\bf 2},
191--205 (1991)

\bibitem{Carlen} Carlen E.: Conservative diffusions, {\em Commun. Math.
Phys.} {\bf 94} (1984), no. 3, 293--315.

\bibitem{Carlen2} Carlen E.: Existence and Sample Path
Properties of the Diffusions in Nelson's Stochastic Mechanics, in:
{\em Lecture Notes in Mathematics, Vol. 1158}, (1985) 25--51,
Springer, Berlin, Heidelberg, New York.

\bibitem{Carlen3} Carlen E.: Progress and Problems in Stochastic Mechanics, in {\em
Stochastic Methods in Mathematical Physics.}, (1989), Word
Scientific, Singapore.


\bibitem{CarlenN} Carlen  E.: Stochastic Mechanics: a Look Back and a Look Ahead,
in: {\em Diffusion, Quantum Theory and Radically Elementary
Mathematics}, chapter 5 (2006), Princeton University Press,
Princeton.

\bibitem{Carlen4} Carlen E.A., Carvalho M.C, Le Roux J., Loss M., Villani C.: Entropy and chaos in the Kac model, {\em
Kinet. Relat. Models} {\bf 3}, 1 (2010), 85--122.


\bibitem{CS} Cherny, A.Y., Shanenko, A.A.: The kinetic and
interaction energies of a trapped Bose gas: Beyond the mean field
\textit{Phys. Lett. a} {\bf 293}, 287 (2002)


\bibitem{Cornell} Cornell E.A., Wieman C.E.: Bose-Einstein condensation in a dilute gas:
the first 70 years and some recent experiments (Nobel Lecture)
{\em Chemphyschem} {\bf 3} (2002), no. 6, 473--493.

\bibitem{Csiszar} Csiszar I.: Information type measures of
difference of probability distributions and indirect observations
{\em Studia Sci. Math. Hungar} {\bf 2} (1967), 299--318.

\bibitem{DPos} Dell'Antonio G., Posilicano A.: Convergence of
Nelson Diffusions, {\em Comm. Math. Phys.} {\bf 141} (1991), no.
3, 559--576.

\bibitem{DeLu} Devroye D., Lugosi G.:{\em Combinatorial Methods in
Density Estimation}, Springer Science\&Business Media, 2012.

\bibitem{DeVecchi} De Vecchi, F.C.: On the strictly convex functionals convergence (preprint
2015)

\bibitem{DeVU} De Vecchi, F.C., Ugolini, S.: An entropy approach
to Bose-Einstein Condensation, {\em Comm. on Stoch. Analysis
(COSA)} {\bf 8} (2014), no. 4, 517--529.








\bibitem{Erdos} Erd\"{o}s L., Schlein B. and Yau H.T.: Rigorous
derivation of the Gross-Pitaevskii equation, {\em Phys. Rev.
Lett.} {\bf 98} (2007), no. 4, 040404, 1--4.



\bibitem{Follmer} F\"{o}llmer, H.: Random Fields and Diffusion
Processes, in: {\em Lecture Notes in Mathematics} {\bf 1362},
(1988) 101--203, Springer, Berlin.

\bibitem{Fukushima} Fukushima, M.: {\em Dirichlet Forms and
Markov Processes}, North-Holland, Amsterdam, 1980.

\bibitem{Fukushima2} Fukushima, M., Oshima, Takeda, : {\em Dirichlet Forms and
Markov Processes}, North-Holland, Amsterdam, 1980.

\bibitem{Gross} Gross E.P.: Structure of a quantized vortex in
boson system, {\em Nuovo Cimento} {\bf 20} (1961), no. 3,
454--477.


\bibitem{HaMi} Hauray, M. and Mischler, S.: On Kac's chaos and
related problems, {\em Journal of Functional Analysis} {\bf 16}
(2014), no. 7, 1423--1466.

\bibitem{Kac} Kac M.: Foundations of kinetic theory, in: {\em
Proceedings of the Third Berkeley Symposium on Mathematical
Statistics and Probability} {\bf 3}, (1956) 171--197, University
of California Press, Berkeley and Los Angeles.

\bibitem{KeDr} Ketterle W., van Druten N.J.: Evaporative
Cooling of Trapped Atoms, in: {\em Advances in Atomic, Molecular
and Optical Physics} {\bf 37}, (1996) 181--236, Academic Press, S.
Diego.




\bibitem{Kullback} Kullback S.: A lower bound for discrimination
in terms of variation, {\em IEEE Trans. Infor. Theory} {\bf 4},
126--127.

\bibitem{LiebYng}  Lieb E. H., Yngvason J.: Ground State Energy of the Low Density Bose
Gas, {\em Phys. Rev. Lett.} {\bf  80} (1998), no. 12, 2504--2507.


\bibitem{Lieb} Lieb E. H., Seiringer R. and Yngvason J.: Bosons in a trap: a rigorous
derivation of the Gross-Pitaevskii energy functional, {\em Phys.
Rev. A} {\bf  61} (2000), 043602, 1--13.

\bibitem{Lieb2} Lieb E. H. and Seiringer R.: Proof of Bose-Einstein condensation for dilute trapped gases
{\em Phys. Rev. Lett.} {\bf  88} (2002), 170409, 1--4.

\bibitem{LiebBook}  Lieb E. H., Seiringer R., Solovej J. P. and Yngvason J.: {\em The Mathematics of the Bose Gas and its Condensation},
Birkh\"{a}user Verlag, Basel, 2005.






\bibitem{Lions} Lions, P.,L.: Two geometrical Properties of
Solutions of Semilinear Problems. {\em Applicable Analysis} {\bf
12}, 267-272 (1981)


\bibitem{LM} Loffredo M. and Morato L.M.: Stochastic Quantization for a system of N identical interacting Bose particles.
{\em  J Phys. A: Math. Theor} {\bf  40} (2007), no. 30, 8709.


\bibitem{MR} Ma, Z.M., R\"{o}ckner, M.: {\it Introduction to the Theory
of (Non-Symmetric) Dirichlet Forms}. Springer-Verlag (1992)



\bibitem{Michelangeli} Michelangeli A.: {\em Bose-Einstein
Condensation: analysis of problems and rigorous results},
Ph.D.Thesis, SISSA, Italy, 2007.


\bibitem{MU} Morato L.M.,Ugolini S.: Stochastic Description of a Bose-Einstein Condensate {\em Annales Henry Poincar\'e }
{\bf 12} (2011), no. 8, 1601--1612.

\bibitem{MU1} Morato L.M.,Ugolini S.: Localization of relative entropy in Bose-Einstein Condensation of trapped interacting
bosons, in:{\em Seminar on Stochastic Analysis, Random Fields and
Applications VII} {\bf 67}, (2013) 197--210, Birkh\"{a}user,
Basel.


\bibitem{Nelson1} Nelson E.: {\em Dynamical Theories of Brownian
Motion}, Princeton University Press, Princeton, 1967.

\bibitem{Nelson2} Nelson E.: {\em Quantum Fluctuations}, Princeton University
Press, Princeton, 1985.

\bibitem{OttoVillani} Otto, F. and Villani, C.: Generalization of
an inequality by Talagrand, and links with the logarithmic Sobolev
Inequality. {\em J. Funct. Anal.} {\bf 173} (2000), no. 2,
361-400.

\bibitem{Pitae} Pitaevskii L.P.: Vortex lines in an imperfect
Bose gas, {\em Sov. Phys.-JETP} {\bf 13} (1961), 451--454.

\bibitem{PosU} Posilicano A., Ugolini S.: Convergence of
Nelson Diffusions with time-dependent Electromagnetic Potentials,
{\em J. Math. Phys.} {\bf 34} (1993), 5028--5036.

\bibitem{ReedSimon} Reed M., Simon B.: {\em Analysis of Operators, Methods of Modern
Mathematical Physics Vol. IV}, Academic Press, S. Diego, 1978.





\bibitem{StroockVaradhan} Stroock D. W., Varadhan S.R.S.: {\em
Multidimensional Diffusion Processes}. Springer-Verlag, New York
(1979)

\bibitem{Sznitman} Sznitman A.S.: Topics in propagation of
chaos, in \textit{Lecture notes in mathematics} {\bf 1464}, (1991)
164--251, Springer.

\bibitem{Ugolini} Ugolini S.: Bose-Einstein Condensation: a
transition to chaos result, {\it Communications on Stochastic
Analysis} {\bf 6} (2012), no. 4, 565--587.

\bibitem{Villani} Villani, C.: {\em Optimal transport, old and
new}. Springer-Verlag, Berlin (2009)


\end{thebibliography}
\end{document}